\documentclass[12pt]{amsart}
\usepackage{latexsym}
\usepackage{amsthm}
\usepackage{amsmath}
\usepackage{amsfonts}
\usepackage{amssymb}
\usepackage[dvips]{graphicx}
\usepackage{xypic}
\input xy
\xyoption{all}

\newtheorem{theo}{Theorem}[section]
\newtheorem{lemma}[theo]{Lemma}
\newtheorem{propo}[theo]{Proposition}
\newtheorem{assump}[theo]{Assumption}
\newtheorem{defi}[theo]{Definition}
\newtheorem{coro}[theo]{Corollary}
\newtheorem{rem}[theo]{Remark}
\newtheorem{stat}[theo]{Statement}

\newcommand\Ind{\operatorname{Ind}}

\newcommand\id{\operatorname{id}}

\newcommand\Set{\operatorname{\bf Set}}
\newcommand\Emb{\operatorname{\bf Emb}}
\newcommand\Str{\operatorname{\bf Str}}

\newcommand\Lin{\operatorname{\bf Lin}}

\newcommand\colim{\operatorname{colim}}

\newcommand\ca{\mathcal {A}}

\newcommand\cc{\mathcal {C}}

\newcommand\ck{\mathcal {K}}
\newcommand\cl{\mathcal {L}}

\date{November 10, 2014}
 
\begin{document}
\title[Classification theory for accessible categories]
{Classification theory for accessible categories}
\author[M. Lieberman and J. Rosick\'{y}]
{M. Lieberman and J. Rosick\'{y}$^*$}
\thanks{$^*$ Supported by the Grant Agency of the Czech Republic under the grant 
               P201/12/G028.} 
\address{
\newline M. Lieberman\newline
Department of Mathematics and Computer Science\newline
Kalamazoo College\newline
1200 Academy Street\newline
Kalamazoo, MI 49006, USA\newline
mlieberm@kzoo.edu\newline
\newline J. Rosick\'{y}\newline
Department of Mathematics and Statistics\newline
Masaryk University, Faculty of Sciences\newline
Kotl\'{a}\v{r}sk\'{a} 2, 611 37 Brno, Czech Republic\newline
rosicky@math.muni.cz
}
 
\begin{abstract}
We show that a number of results on abstract elementary classes (AECs) hold in accessible categories with concrete directed colimits.  In particular, we prove a generalization of a recent result of Boney on tameness under a large cardinal assumption.  We also show that such categories support a robust version of the Ehrenfeucht-Mostowski construction.  This analysis has the added benefit of producing a purely language-free characterization of AECs, and highlights the precise role played by the coherence axiom.
\end{abstract} 

\maketitle
 
\section{Introduction}
Classical model theory studies structures using the tools of first order logic.  In an effort to develop classical results in more general logics, Shelah introduced abstract elementary classes (AECs), a fundamentally category-theoretic generalization of elementary classes, in which logic and syntax are set aside and the relevant classes of structures are axiomatized in terms of a family of strong embeddings (see \cite{S}; \cite{Ba} contains the resulting theory). Accessible categories, on the other hand, were first introduced by Makkai and Par\'e in response to the same fundamental problem: where earlier work in categorical logic had focused on the structure of theories and their associated syntactic categories, with models a secondary notion, accessible categories represented an attempt to capture the essential common structure of the categories of models themselves.  Despite the affinity of these two ideas, it is only recently that their connections have begun to be appreciated (see \cite{BR} and \cite{L}).  In particular, AECs have been shown to be special accessible categories with directed colimits (i.e. direct limits).  We refine this characterization further, axiomatizing AECs as special accessible categories with {\it concrete} directed colimits; that is, we describe them as pairs $(\ck,U)$ where $\ck$ is an accessible category with directed colimits, and $U:\ck\to\Set$ is a faithful functor into the category of sets that preserves directed colimits.  For such a pair to be an AEC, it must also satisfy three additional conditions: all morphisms must be monomorphisms, the category must have the property of coherence, and it must be iso-full.  The first condition can be obtained without loss of generality (see Remark \ref{re3.2}).  Coherence (sometimes known as the "funny" axiom for AECs) can, surprisingly, be formulated as a property of the functor $U$ (see Definition \ref{def3.1})---our analysis highlights areas in which this hypothesis appears to be indispensable, and those in which it can be set aside entirely.  It is more complicated to formulate iso-fullness as a property of $U$, but this can also be done: see Remark \ref{re3.5} below.

We wish to emphasize that Shelah's Categoricity Conjecture, which is the main test question for AECs, is a property
of a category $\ck$ itself -- one does not need $U$ for its formulation. For a category theorist, this question may appear artificial, but it can
be reformulated as an ``injectivity property": the Categoricity Conjecture is equivalent to the assertion that any $\lambda$-saturated object is $\mu$-saturated for all $\lambda<\mu$ where $\lambda$-saturated
means being injective with respect to morphisms between $\lambda$-presentable objects (with $\lambda$ and $\mu$ regular cardinals). 
This can be quite easily proved assuming the presence of pushouts, which is a very strong amalgamation property, never present in an AEC---induced mappings from a putative pushout object will not generally be monomorphisms, hence cannot be $\ck$-morphisms. 
Given the weaker hypothesis of amalgamation, one instead relies on constructions involving Galois types, including the element by element construction of $\ck$-morphisms, which often forces one to assume coherence. Grossberg and VanDieren
\cite{GV} have succeeded in showing Shelah's Categoricity Conjecture for successor cardinals in tame AECs. Recently, Boney
\cite{B} proved that, assuming the existence of a proper class of strongly compact cardinals, every AEC is tame. In combination 
with \cite{GV}, this implies that, assuming the existence of a proper class of strongly compact cardinals, Shelah's Categoricity Conjecture in a successor cardinal
is true for AECs. We will introduce tameness for accessible categories with concrete directed colimits and will show that, assuming
the existence of a proper class of strongly compact cardinals, every accessible category with concrete directed colimits is tame. In fact this generalization
of Boney's theorem follows from an old result of Makkai and Par\'e (\cite{MP}, 5.5.1) about accessible categories. 
In addition, we begin the process of extending a fragment of stability theory from AECs to accessible categories with concrete directed colimits, a process which leads to several useful insights.  First, we note that such categories admit a robust EM-functor---the existence of such a functor 
in an abstract elementary class is one of many results that flow from Shelah's Presentation Theorem, which involves both the assumption of coherence and the reintroduction of language into the fundamentally syntax-free world of AECs. This is an added benefit: the current study highlights areas 
in which coherence can be dispensed with---the existence of EM-models being a particularly noteworthy example---and those where it appears to be essential---without coherence, arguments involving the element-by-element construction of morphisms become problematic, if not impossible.  

On the one hand, this supports the contention that AECs strike the appropriate balance between structure and generality, and are thus ideally suited for the development of abstract classification theory.  On the other, it sheds light on the extent to which classification theory can be developed in a more general (and more category-theoretically natural) setting.

The authors wish to acknowledge the anonymous referee, whose feedback has led to noteworthy improvements in this paper.  Thanks are also due to John Baldwin, who provided valuable input at many stages of the writing process.

\section{Accessible categories with directed colimits}

Accessible categories were introduced in \cite{MP} as categories closely connected with categories of models of infinitary ($L_{\kappa,\lambda}$)
theories. Roughly speaking, an accessible category is one that is closed under certain directed colimits, and whose objects can be built via certain directed colimits of a set of small objects.  To be precise, we say that a category $\ck$ is $\lambda$-\textit{accessible}, $\lambda$ a regular cardinal, if it closed under $\lambda$-directed colimits (i.e. colimits indexed by a $\lambda$-directed poset) and contains, up to isomorphism, a set $\ca$ 
of $\lambda$-presentable objects such that each object of $\ck$ is a $\lambda$-directed colimit of objects from $\ca$. Here $\lambda$-presentability functions as a notion of size that makes sense in a general, i.e. non-concrete, category: we say an object $K$ is $\lambda$-\textit{presentable} if its hom-functor $\ck(K,-):\ck\to\Set$ preserves $\lambda$-directed colimits. Put another way, $K$ is $\lambda$-presentable if for any morphism $f:K\to M$ with $M$ a $\lambda$-directed colimit $\langle \phi_\alpha:M_\alpha\to M\rangle$, $f$ factors essentially uniquely through one of the $M_\alpha$, i.e. $f=\phi_\alpha f_\alpha$ for some $f_\alpha:K\to M_\alpha$.  We typically write {\it finitely accessible} and {\it finitely presentable} in place of $\omega$-accessible and $\omega$-presentable, respectively: the category of groups, for example, is finitely accessible and an object is finitely presentable if and only if it is finitely presented in the usual sense.
Detailed treatments of general accessible categories can be found in \cite{MP} or \cite{AR}---we here develop the theory of accessible categories satisfying progressively stronger and more recognizably model-theoretic conditions. 

For each regular cardinal $\kappa$, an accessible category contains, up to isomorphism, only a set of $\kappa$-presentable objects.
Any object $K$ of a $\lambda$-accessible category is $\kappa$-presentable for some regular cardinal $\kappa$. Given an object $K$, the smallest cardinal $\kappa$ such that $K$ is not $\kappa$-presentable is called the \textit{presentability rank} of $K$.  This notion is particularly well behaved in case the accessible category has all directed colimits, not merely the $\lambda$-directed ones (see \cite{BR} 4.2):

\begin{stat}\label{st2.0}
{\em If $\ck$ has directed colimits then the presentability rank $\kappa$ of any object $K$ is a successor cardinal, i.e. $\kappa=|K|^+$. We say 
that $|K|$ is the \textit{size} of $K$.
}
\end{stat} 

These sizes play a role analogous to that of the cardinalities of underlying sets in AECs.

If $\lambda<\mu$, then any $\lambda$-directed colimit is $\mu$-directed, but a $\lambda$-accessible category $\ck$ need not be $\mu$-accessible. 
For this one needs a finer relation, $\lambda\trianglelefteq\mu$ (see \cite{AR} 2.12). This technical issue disappears, however, if $\ck$ has all directed colimits: we can replace $\trianglelefteq$ with $\leq$ (see \cite{BR} 4.1). Moreover, if a functor
$F:\ck\to\cl$ between $\lambda$-accessible categories with directed colimits preserves $\lambda$-directed colimits and $\lambda$-presentable objects
then $F$ also preserves $\mu$-presentable objects for all $\mu>\lambda$. Without directed colimits, we have this only if $\lambda\trianglelefteq\mu$. 

The process of developing a fragment of model theory within accessible categories began in \cite{R}.  Certain essential properties of AECs are fundamentally diagrammatic, and are adopted without change.  For example, we say that an accessible category $\ck$ satisfies the \textit{amalgamation property} if any pair of morphisms $f:K\to L$, $g:K\to M$ can be completed to a commutative
diagram 
$$
\xymatrix@=4pc{
L \ar[r]^{} & N \\
K \ar [u]^{f} \ar [r]_{g} & M \ar[u]_{}
}
$$
We say $\ck$ satisfies the \textit{joint embedding property} if for any two objects $K,L$ in $\ck$ there is an object $N$ with morphisms $K\to N$
and $L\to N$.

More importantly, \cite{R} introduces a number of essentially model-theoretic notions and arguments, particularly those involving saturation and weak stability.  For a regular cardinal $\lambda$, an object $K$ of a category $\ck$ is said to be $\lambda$-\textit{saturated} if it is injective with respect to morphisms between $\lambda$-presentable objects. This means that for any morphisms 
$f:A\to K$ and $g:A\to B$ where $A$ and $B$ are $\lambda$-presentable there is a morphism $h:B\to K$ such that $hg=f$. In the category of models of
a (finitary) first-order theory and elementary embeddings, this notion coincides with the usual one. In an abstract elementary class, it coincides
with $\lambda$-model homogeneity. 

\cite{R} also introduces \textit{weak} $\lambda$-\textit{stability} for accessible categories, which forms the crux of the argument for the existence of monster objects---large, highly saturated objects---in accessible categories with directed colimits and amalgamation.  We give the argument in some detail, as it has not heretofore been developed in this generality.  As in AECs, we must assume the existence of arbitrarily large cardinals $\lambda$ with $\lambda^{<\lambda}=\lambda$, which follows from either GCH or the existence of a proper class of strongly inaccessible cardinals.  As in discussions of monster models in the AEC literature, too, we note that monster objects are used largely as a convenient shorthand, and can typically be written out of proofs, albeit at some cost in length and comprehensibility.

The central idea is that any accessible category with directed colimits and amalgamation is weakly stable in many cardinalities, and that weak stability in sufficiently large cardinals guarantees the existence of highly saturated objects.  In particular, in any weakly $\mu$-stable $\mu$-accessible
category with directed colimits and the amalgamation property, any $\mu^+$-presentable object admits a morphism to a $\mu^+$-presentable $\mu$-saturated
object (see \cite{R} Remark 4(2)).  Crucially, Proposition 2 asserts that if $\ck$ is a $\lambda$-accessible category and $\lambda\trianglelefteq\mu$ is a regular cardinal greater then the number of morphisms between $\lambda$-presentable objects of $\ck$ then $\ck$ is weakly $\mu^{<\mu}$-stable.  If $\ck$ has all directed limits, $\lambda\trianglelefteq\mu$ can
be weakened to $\lambda\leq\mu$, and the two results would combine to give a ready supply of monster objects.  There is an important gap in that proof of Proposition 2, however, which we now fill: as long as $\mu^{<\mu}=\mu$, the argument goes through as written.  Provided we assume that there are a proper class of such cardinals, we have the promised abundance of monsters.  To be precise, after this correction, \cite{R} yields the following 

\begin{stat}\label{st2.1}
{
\em 
Let $\ck$ be an accessible category with directed colimits and the amalgamation property and assume that there are arbitrarily large regular cardinals $\lambda$
such that $\lambda^{<\lambda}=\lambda$. Then there are arbitrarily large regular cardinals $\lambda$ such that any $\lambda^+$-presentable object $K$ admits 
a morphism $g:K\to L$ to a $\lambda^+$-presentable $\lambda$-saturated object $L$.
}
\end{stat} 
 
In a $\lambda$-accessible category having directed colimits and the joint embedding property, any two $\lambda^+$-presentable $\lambda$-saturated objects 
are isomorphic (see \cite{R} Theorem 2). The proof of this theorem yields the following 

\begin{stat}\label{st2.2}
{
\em
Let $\ck$ be a $\lambda$-accessible category with directed colimits and the joint embedding property and $g_1:K_1\to L$, $g_2:K_2\to L$ 
be morphisms from $\lambda^+$-presentable objects to $\lambda^+$-presentable $\lambda$-saturated objects. Then any morphism $f:K_1\to K_2$ extends 
to an isomorphism $s:L\to L$ such that $sg_1=g_2f$.
}
\end{stat}

\begin{rem}\label{re2.3}
{
\em
The $\lambda$-saturated objects of size $\lambda$ thus obtained are our {\it monster objects}.

}
\end{rem}

We say that a category is \textit{large} if it is not equivalent to a small category. For accessible categories, this entails that there exist objects
with arbitrarily large presentability ranks---much like, in the model-theoretic context, the assumption of arbitrarily large models. For large categories of models of $L_{\kappa,\omega}$, downward L\"{o}wenheim-Skolem theorem
implies the existence of models of all sizes starting from some cardinal. Motivated by this, \cite{BR} called a category $\ck$ \textit{LS-accessible}
if it is an accessible category for which there is a cardinal $\lambda$ such that $\ck$ has objects of all sizes $\mu\geq\lambda$. Any LS-accessible
category is large. The smallest cardinal $\lambda$ such that $\ck$ is $\mu$-accessible for all regular cardinals $\mu\geq\lambda$ and has objects 
of all sizes $\mu\geq\lambda$ will be denoted $\lambda_\ck$. 

Any large finitely accessible category is LS-accessible. This follows from \cite{MP} 3.4.1 which shows that any large finitely accessible category $\ck$ admits an analogue of the Ehrenfeucht-Mostowski functor:
a faithful functor $E:\Lin\to\ck$ preserving directed colimits, where $\Lin$ is the category of linearly ordered sets and order preserving injective mappings.
In fact, following \cite{BR} 4.4, $E$ preserves sizes starting from some cardinal, thereby ensuring LS-accessibility.  It is not known whether every large accessible category with directed colimits admits an EM-functor (and is consequently LS-accessible), although \cite{BR} contains a couple of results in this direction. We give another partial result here, proving that this is true if, as in AECs, the morphisms in $\ck$ are assumed to be monomorphisms.  To begin:

\begin{propo}\label{prop2.4}
Let $\ck$ be an accessible category with directed colimits whose morphisms are monomorphisms. Then there is a faithful functor $F:\cl\to\ck$ preserving
directed colimits with $\cl$ finitely accessible. 
\end{propo}
\begin{proof} Let $\ck$ be a $\lambda$-accessible category with directed colimits and $\cc$ be its representative full subcategory consisting
of $\lambda$-presentable objects; this means that any $\lambda$-presentable object of $\ck$ is isomorphic with $C\in\cc$. Let $\cl$ be a 
free completion of $\cc$ under directed colimits. Then $\cl$ is finitely accessible (see \cite{AR} 2.26) and, since $\ck$ has directed colimits,
there is a functor $F:\Ind(\cc)\to\ck$ preserving directed colimits. It suffices to show that $F$ is faithful.
Consider two distinct morphisms $f,g:K\to L$ in $\ck$. Then there are morphisms $f',g':A\to B$ in $\cc$ and morphisms
$u:A\to K$, $v:B\to L$ such that $fu=vf'$ and $gu=vg'$. Since $F$ is faithful on $\cc$ (because it is the identity on $\cc$),
$Ff',Fg'$ are distinct. Since $Fv$ is a monomorphism, $Ff,Fg$ are distinct.
\end{proof}

\begin{rem}\label{re2.5}
{
\em
Any finitely accessible category $\cl$ is $L_{\kappa,\omega}$-axiomatizable and the functor $F$ is surjective on objects.  Consequently, Proposition \ref{prop2.4}
may be regarded as an analogue of Shelah's Presentation Theorem for AECs: omitting types can be expressed in $L_{\kappa,\omega}$, and the reduct involved in the Presentation Theorem is surjective on models in the AEC.
}
\end{rem}

\begin{coro}\label{cor2.6}
Any large accessible category with directed colimits whose morphisms are monomorphisms is LS-accessible.
\end{coro}
\begin{proof}
If $\ck$ is a large accessible category with directed colimits whose morphisms are monomorphisms then $\cl$ in \ref{prop2.4} is large as well.
Following \cite{MP}, 3.4.1, there is a faithful functor $E:\Lin\to\cl$ preserving directed colimits. The composition $FE:\Lin\to\ck$ is faithful and preserves directed colimits. Following \cite{BR}, 4.4, $FE$ preserves sizes starting from some cardinal. Thus $\ck$ is LS-accessible. 
\end{proof}

\begin{rem}\label{re2.7}
{
\em
This yields the promised EM-functor: given any accessible category with directed colimits $\ck$ whose morphisms are monomorphisms, there is a faithful functor from $\Lin$ to $\ck$ that preserves directed colimits and sufficiently large sizes.  Note that the proof of its existence relies neither on coherence nor on the Presentation Theorem, in contrast to AECs.   
}
\end{rem}


\section{Accessible categories with concrete directed colimits}

We now begin to add the structure necessary to develop a meaningful classification theory for accessible categories, beginning with the essential step of systematically associating sets with the objects of our categories.  In particular, we say that $(\ck,U)$ is an \textit{accessible category with concrete directed colimits} if $\ck$ is an accessible category with directed colimits 
and $U:\ck\to\Set$ is a faithful functor to the category of sets preserving directed colimits. 

Any large accessible category with concrete directed colimits is LS-acce\-ssib\-le (see \cite{BR} 4.12).  While one would not go far wrong in thinking of $UK$ as the ``underlying set'' of an object $K$, it is important to note that the size of $K$ in $\ck$, as defined in Statement \ref{st2.0}, need not correspond to the cardinality of $UK$.  That is, sizes need not be preserved by $U$.  If, however, $U$ \textit{reflects split epimorphisms}---whenever $U(f)g=\id$ then $fg^\prime=\id$ 
for some $g^\prime$---it does in fact preserve sizes starting from some cardinal (see \cite{BR} 4.3 and 3.7). The smallest cardinal $\lambda\geq\lambda_\ck$
with this property will be denoted by $\lambda_U$. This means that $\ck$ has objects of all sizes $\mu\geq\lambda_U$ and these sizes are equal
to cardinalities of underlying sets.

There is a more familiar condition that will achieve the same effect: it suffices to assume $(\ck,U)$ satisfies a generalization of the coherence axiom for AECs.

\begin{defi}\label{def3.1}
{
\em
Let $(\ck,U)$ be an accessible category with concrete directed colimits. We say that $\ck$ is \textit{coherent} if for each commutative triangle 
$$
\xymatrix@=3pc{
UA \ar[rr]^{U(h)}
\ar[dr]_{f} && UC\\
& UB \ar[ur]_{U(g)}
}
$$
there is $\overline{f}:A\to B$ in $\ck$ such that $U(\overline{f})=f$.
}
\end{defi}

If all morphisms in $\ck$ are monomorphisms, $U$ reflects split epimorphisms if and only if it is \textit{conservative}, i.e. if it reflects isomorphisms
(\cite{BR} 3.5). Coherence is a much stronger condition---if we are concerned only with the preservation of sizes, we can get away with less.

Without any loss of generality, we can pass from accessible categories with concrete
directed colimits to accessible categories with concrete directed colimits whose morphisms are \textit{concrete monomorphisms}, i.e. monomorphisms preserved 
by $U$. 
 
\begin{rem}\label{re3.2}
{
\em
(1) Let $(\ck,U)$ be an accessible category with concrete directed colimits. Consider a pullback
$$
\xymatrix@=4pc{
\ck \ar[r]^{U} & \Set \\
\ck_0 \ar [u]^{G} \ar [r]_{U_0} &
\Emb(\Set) \ar[u]_{}
}
$$
where $\Emb(\Set)$ is the category of sets and monomorphisms. Then $(\ck_0,U_0)$ is an accessible category with concrete directed colimits
(see \cite{MP} 5.1.6 and 5.1.1) whose morphisms are monomorphisms preserved by $U_0$. The functor $G$ is faithful and preserves directed colimits.
Moreover, $(\ck_0,U_0)$ is coherent provided that $(\ck,U)$ is coherent. Thus, for $(\ck,U)$ coherent, $G$ preserves presentability ranks starting 
from some cardinal.

(2) If $\ck$ is large then $\ck_0$ is large as well and, following \ref{re2.7}, there is an EM-functor $E:\Lin\to\ck_0$. The composition $GE$ 
is the EM-functor for $\ck$.
}
\end{rem}

\begin{rem}\label{re3.3}
{
\em
Assume that there are arbitrarily large regular cardinals $\lambda$ such that $\lambda^{<\lambda}=\lambda$.
Let $(\ck,U)$ be a large accessible category with concrete directed colimits, the amalgamation property and the joint embedding property and whose
morphisms are concrete monomorphisms. Let $K$ be a $\lambda$-saturated $\lambda^+$-presentable object with $\lambda_U\leq\lambda$. We will show
that $|K|=\lambda$. Thus, following \ref{st2.1}, $\ck$ contains arbitraily large saturated objects.

Let $A$ be an object of size $\lambda$. Then $A$ is a colimit of a smooth chain of cardinality $\lambda$ consisting of $\lambda$-presentable objects
(see \cite{R1}, Lemma 1). Since $K$ is $\lambda$-saturated, there is a morphism $h:A\to K$. Since $h$ is a concrete monomorphism, we have
$$
|K|=|UK|\geq|UA|=|A|=\lambda.
$$
Thus $|K|=\lambda$. 
}
\end{rem}


\begin{lemma}\label{le3.4}
Any abstract elementary class is a coherent accessible category with concrete directed colimits. 
\end{lemma}
\begin{proof}
Let $\ck$ be an abstract elementary class. Then there is a finitary signature $\Sigma$ such that $\ck$ is a subcategory of the category
$\Emb(\Sigma)$ of $\Sigma$-structures whose morphisms are substructure embeddings. Moreover, the inclusion $\ck\to\Emb(\Sigma)$ preserves
directed colimits and, whenever $fg$ and $f$ are morphisms in $\ck$ then $g$ is a morphism in $\ck$. Since the forgetful functor 
$U:\Emb(\Sigma)\to\Set$ is coherent and preserves directed colimits, its restriction to $\ck$ has the same properties.
\end{proof}

The precise relationship between AECs and accessible ca\-te\-go\-ries was clarified in \cite{L} and \cite{BR}. Recall that an accessible category
with directed colimits whose morphisms are monomorphisms is equivalent to an abstract elementary class if and only if it admits a coherent iso-full
embedding into a finitely accessible category preserving directed colimits and monomorphisms (\cite{BR})---roughly speaking, an ambient category of structures. We will improve on this characterization by giving an entirely language-independent 
description of AECs, axiomatized entirely in terms of properties of $\ck$ and $U:\ck\to\Set$. Among other things, this shows that a coherent accessible category with concrete directed colimits whose morphisms are monomorphisms preserved by $U$ is very close to being an abstract elementary class. 

To achieve this, we note that $U:\ck\to\Set$ itself determines a canonical signature $\Sigma_\ck$ for working with $(\ck,U)$:
 
\begin{rem}\label{re3.5}
{
\em
Let $(\ck,U)$ be a $\lambda$-accessible category with concrete directed colimits whose morphisms are monomorphisms preserved by $U$. For a finite 
cardinal $n$ we denote by $U^n$ the functor
$$
\Set(n,U(-)):\ck\to\Set.
$$
Directed colimits preserving subfunctors of $U^n$ will be called \textit{finitary relation symbols interpretable in} $\ck$ and natural
transformations $h:U^n\to U$ will be called \textit{finitary function symbols interpretable in} $\ck$ (cf. \cite{R0}). Since they are both
determined by their restrictions to the full subcategory $\ck_\lambda$ of $\ck$ consisting of $\lambda$-presentable objects, there is only
a set of such symbols. They form the signature $\Sigma_\ck$, which we call the \textit{canonical signature} of $\ck$. We get
a \textit{canonical functor} $G:\ck\to\Str(\Sigma_\ck)$ into $\Sigma_\ck$-structures where morphisms are homomorphisms. 

Given a bijection $f:UA\to UB$, we get bijections $f^n:(UA)^n\to (UB)^n$. Assume that $R(f)$ is a bijection for each $R$ and $h_Bf^n=fh_A$
for each $h$ above. This means that $f:GA\to GB$ is an isomorphism. We say that $\ck$ is \textit{iso-full} if $f=U(\overline{f})$ for each such $f$.

There is a largest subsignature $\Sigma_0$ of $\Sigma_\ck$ such that the induced functor $G_0:\ck\to\Str(\Sigma_0)$ has image in $\Emb(\Sigma_0)$ (because this property is closed under union of subsignatures and is satisfied by the empty signature). Now, $(\ck,U)$ is an abstract elementary class if and only if $G_0$ is iso-full.
}
\end{rem}

\begin{rem}\label{re3.6}
{
\em
(1) If $(\ck,U)$ is an abstract elementary class then $\lambda_U$ is its L\"{o}wenheim-Skolem number.

(2) We say that an accessible category $(\ck,U)$ with concrete directed colimits is \textit{finitely coherent} if for each $f:UA\to UB$ 
with the property that for any finite set $X$ and any mapping $a:X\to UA$ there are $h:B\to C$ and $g:A\to C$ with $U(h)fa=U(g)a$, $f$ carries a $\ck$-morphism,
i.e., $f=U(\overline{f})$ for some $\overline{f}:A\to B$.

Any finitely coherent accessible category with concrete directed colimits is coherent. We now show that an abstract elementary class  with the amalgamation
property is finitary in the sense of \cite{HK} if and only if the corresponding $(\ck,U)$ is finitely coherent. 

Let $\ck$ be finitary and consider  $f:UA\to UB$ such that for any finite set $X$ and any mapping $a:X\to UA$ there are $h:B\to C$ and $g:A\to C$ 
with $U(h)fa=U(g)a$. Then $f$ carries a $\ck$-morphism and thus $(\ck,U)$ is finitely coherent. Conversely, let $(\ck,U)$ be finitely coherent.
Then $\ck$ is finitary in the sense of \cite{K}, which is equivalent, assuming the amalgamation property, to being finitary in the sense of \cite{HK}.
}
\end{rem}
 

\section{Galois types}
Galois types were introduced for AECs by Shelah (cf. \cite{B}) and his definition can be translated directly into the framework of accessible categories
with concrete directed colimits.

\begin{defi}\label{def4.1}
{
\em
Let $(\ck,U)$ be an accessible category with concrete directed colimits. A \textit{type} is a pair $(f,a)$ where $f:M\to N$ and $a\in UN$.

The types $(f_0,a_0)$ and $(f_1,a_1)$ are called equivalent if there are morphisms $h_0:N_0\to N$ and $h_1:N_1\to N$ such that 
$h_0f_0=h_1f_1$ and $U(h_0)(a_0)=U(h_1)(a_1)$.
}
\end{defi}

Assuming the amalgamation property we get an equivalence relation.  As in AECs, the resulting equivalence classes are called \textit{Galois types}.  In the remainder of the paper, we make the following blanket assumption: 

\begin{assump}\label{blanket}
{\em All categories satisfy the amalgamation and joint embedding properties.}\end{assump}

In many of the results that follow, beginning with the following lemma, we invoke the existence of a monster objects in our category.  In each case, we clearly indicate the large cardinal assumption required to ensure their existence: per Remark \ref{re2.3}, it suffices to assume there is a proper class of cardinals $\lambda$ with $\lambda^{<\lambda}=\lambda$.

\begin{lemma}\label{le4.2}
Suppose there is a proper class of cardinals $\lambda$ with $\lambda^{<\lambda}=\lambda$.  Let $(\ck,U)$ be an accessible category with concrete directed colimits. Then 
types $(f_0,a_0)$  and $(f_1,a_1)$ are equivalent if and only if there is a $\lambda$-saturated, $\lambda^+$-presentable object $L$ (for some regular
cardinal $\lambda$), morphisms $g_0:N_0\to L$, $g_1:N_1\to L$ and an isomorphism $s:L\to L$ such that $sg_0f_0=g_1f_1$ and $U(sg_0)(a_0)=U(g_1)(a_1)$.
\end{lemma}
\begin{proof}
Sufficiency is evident because $sg_0,g_1$ provide $h_0,h_1$. Assume that the types are equivalent via $h_0:N_0\to N$ and $h_1:N_1\to N$. 
There is a regular cardinal $\lambda$ such that $\ck$ is $\lambda$-accessible, weakly $\lambda$-stable and $M,N_0,N_1,N$ are $\lambda$-presentable. 
Following \ref{st2.1} and \ref{st2.2}, there is a $\lambda$-saturated and $\lambda^+$-presentable object $L$ equipped with 
morphisms $g_0:N_0\to L$, $g_1:N_1\to L$ and $g:N\to L$ and isomorphisms $s_0,s_1:L\to L$ such that $g_0f_0=g_1f_1$, $s_0g_0=gh_0$ and $s_1g_1=gh_1$. 
Then $U(s_0g_0)(a_0)=U(s_1g_1)(a_1)$ and thus $U(s_1^{-1}s_0g_0)(a_0)=U(g_1)(a_1)$.
\end{proof}

\begin{rem}\label{re4.3}
{
\em
In fact, in the situation of \ref{le3.4}, $s$ exists for any given $g_0$ and $g_1$.
}
\end{rem}

\section{Tameness}

Like Galois types, tameness can be easily redefined in the framework of accessible categories with concrete directed colimits.
 
\begin{defi}\label{def5.1}
{\em 
Let $(\ck,U)$ be an accessible category with concrete directed colimits and $\kappa$ a regular cardinal. We say that $\ck$ is $\kappa$-\textit{tame} 
if for two non-equivalent types $(f,a)$ and $(g,b)$ there is a morphism $h:X\to M$ with $X$ $\kappa$-presentable such that the types $(fh,a)$ 
and $(gh,b)$ are not equivalent.

$\ck$ is called \textit{tame} if it is $\kappa$-tame for some regular cardinal $\kappa$.
}
\end{defi}

Recall that a cardinal $\kappa$ is called \textit{strongly compact} if, for any set $S$, every $\kappa$-complete filter over $S$ can be extended
to a $\kappa$-complete ultrafilter over $S$ (see \cite{J}).  In what fo\-llows $(C)$ will
denote the existence of a proper class of strongly compact cardinals. 

\begin{theo}\label{th5.2}
Assuming $(C)$, any accessible category with concrete directed colimits is tame.
\end{theo}
\begin{proof}
Let $\ck$ be an accessible category with concrete directed colimits and let $\cl_2$ be the category of quadruples $(f_0,f_1,a_0,a_1)$ where
$f_0:M\to N_0$, $f_1:M\to N_1$, $a_0\in UN_0$ and $a_1\in U(N_1)$. Let $\cl_1$ be the category of configurations $(f_0,f_1,a_0,a_1,h_0,h_1)$ 
from \ref{def4.1}. Then both $\cl_1,\cl_2$ are accessible and the forgetful functor $G:\cl_1\to\cl_2$ is accessible as well. 
It is easy to see that the full image of $G$ is a sieve, i.e., for a morphism $(u,v):(g_0,g_1)\to (f_0,f_1)$ with $(f_0,f_1)\in G(\cl_1)$ we have $(g_0,g_1)\in G(\cl_1)$. Following \cite{MP} 5.5.1, the full image of $G$ is $\kappa$-accessible and closed under $\kappa$-directed colimits in $\cl_2$ for some strongly compact cardinal $\kappa$. We will show that $\ck$ is $\kappa$-tame.

Consider $(f_0,f_1,a_0,a_1)$ such that the types $(f_0,u,a_0)$ and $(f_1u,a_1)$ are equivalent for any $u:X\to M$, $X$ $\kappa$-presentable. 
Thus all qua\-drup\-les $(f_0u,f_1u,a_0,a_1)$ belong to the full image of $G$ and, since $(f_0,f_1,a_0,a_1)$ is their $\kappa$-filtered colimit, it belongs
to this full image as well. Thus the types $(f_0,a_0)$ and $(f_1,a_1)$ are equivalent. Hence $\ck$ is $\kappa$-tame. 
\end{proof}

As a consequence, we get the main result of \cite{B}.

\begin{coro}\label{cor5.3} 
Assuming $(C)$, any AEC is tame.
\end{coro}

As noted in \cite{BaS} and \cite{B}, the sensitivity of Theorem \ref{th5.2} and Corollary \ref{cor5.3} to set theory is genuine: assuming $V=L$, the AEC of exact sequences constructed in Section $2$ of \cite{BaS} is not tame.

\section{Saturation}
 
\begin{defi}\label{def6.1}
{
\em 
Let $(\ck,U)$ be an accessible category with concrete directed colimits. We say that a type $(f,a)$ where $f:M\to N$ is \textit{realized} in $K$ if there
is a morphism $g:M\to K$ and $b\in U(K)$ such that $(f,a)$ and $(g,b)$ are equivalent.


Let $\lambda$ be a regular cardinal. We say that $K$ is $\lambda$-\textit{Galois saturated} if for any $g:M\to K$ where $M$ is $\lambda$-presentable 
and any type $(f,a)$ where $f:M\to N$ there is $b\in U(K)$ such that $(f,a)$ and $(g,b)$ are equivalent.
}
\end{defi}

In AECs with amalgamation, of course, this corresponds to the usual definition.  In that case, $\lambda$-Galois saturation is equivalent to $\lambda$-model homogeneity or, in the language of \cite{R} and Section 2 above, $\lambda$-saturation.  This equivalence also holds in our more general context, by an argument thinly generalizing that of \cite{B}:

\begin{propo}\label{prop6.2} 
Suppose there is a proper class of cardinals $\lambda$ with $\lambda^{<\lambda}=\lambda$.  Let $(\ck,U)$ be a coherent large accessible category with concrete directed colimits whose morphisms are concrete monomorphisms and let $\lambda$ be a sufficiently large regular cardinal. Then $K$ is $\lambda$-Galois saturated if and only if it is $\lambda$-saturated.
\end{propo}
\begin{proof}
Let $K$ be $\lambda$-saturated and consider $g:M\to K$ with $M$ $\lambda$-pre\-sen\-table. Consider $(f,a)$ where $f:M\to N$. There is
a $\lambda$-pre\-sen\-table object $N_0$ and a factorization of $f$ over $f_0:M\to N_0$ such that $a\in U(N_0)$. The types $(f,a)$ and
$(f_0,a)$ are equivalent. Since $K$ is $\lambda$-saturated, there is a morphism $g_0:N_0\to K$ such that $g_0f_0=g$. Then the types
$(f_0,a)$ and $(g,U(g_0)(a))$ are equivalent. Thus $K$ is $\lambda$-Galois saturated.

Conversely, let $K$ be $\lambda$-Galois saturated, $h:M\to N$ be a morphism between $\lambda$-presentable objects and $f:M\to K$ a morphism. 
Following \ref{st2.1}, there is a cardinal $\mu$ such that $M$, $N$ and $K$ are $\mu^+$-presentable and equipped with morphisms $g_1:N\to L$
and $g_2:K\to L$ to a $\mu$-saturated $\mu^+$-presentable object $L$. We proceed, roughly speaking, as in the proof of Theorem 8.14 in \cite{Ba}: 
enumerate $U(N)\setminus U(h)(U(M))=\langle a_i\,|\,i<\alpha\rangle$, where $\alpha<\lambda$. We construct, inductively:\\
1. a smooth chain $(m_{ij}: M_i\to M_j)_{i\leq j\leq \alpha}$ of $\lambda$-presentable objects $M_i$ with $M_0=M$ and morphisms $f_i:M_i\to K$,
$u_i:M_i\to L$ for $i\leq\alpha$ such that $f_0=f$, $f_jm_{ij}=f_i$, $u_0=g_1h$ and $u_jm_{ij}=u_i$ for $i\leq j\leq\alpha$, and\\
2. mappings $t_i:U(h)(U(M))\cup \{a_k|k<i\}\to U(M_i)$ for $i\leq\alpha$ such that $t_0=h^{-1}$ restricted to $U(h)(UM)$, $t_iU(h)=U(m_{0i})$, $t_jU(m_{ij})=t_i$ 
and $U(u_i)t_i=U(g_1)$ for $i\leq j\leq\alpha$ (in the last equation, $U(g_1)$ is restricted to the domain of $t_i$).

Suppose we have constructed $M_i$, $f_i$ and $t_i$. Consider the type 
$$(u_i,U(g_1)(a_i)).$$
Since $K$ is $\lambda$-Galois saturated, there is $b\in U(K)$ such that this ty\-pe is equivalent to $(f_i, b)$. 
Following \ref{le4.2}, there is an isomorphism $s:L\to L$ such that $sg_2f_i=u_i$ and $U(sg_2)(b)=U(g_1)(a_i)$. 
There is a $\lambda$-presentable object $M_{i+1}$, an element $c\in U(M_{i+1})$ and morphisms $m_{i,i+1}:M_i\to M_{i+1}$, $f_{i+1}:M_{i+1}\to K$   
such that $U(f_{i+1})(c)=b$ and $f_{i+1}m_{i,i+1}=f_i$. Put $u_{i+1}=sg_2f_{i+1}$. Then 
$$
u_{i+1}m_{i,i+1}=sg_2f_{i+1}m_{i,i+1}=sg_2f_i=u_i.
$$
Let $t_{i+1}U(m_{i,i+1})=t_i$ and $t(a_i)=c$. Then $t_iU(h)=U(m_{0i})$ and, since
$$
U(u_{i+1})(c)=U(sg_2f_{i+1})(c)=U(sg_2)(b)=u(g_1)(a_i),
$$
we have $U(u_{i+1})t_{i+1}=U(g_1)$. In limit steps we take colimits. 
 
Since $U(u_\alpha)t_\alpha=U(g_1)$ and $(\mathcal K,U)$ is coherent, $t_\alpha=U(\overline{t}_\alpha)$ for
$\overline{t}_\alpha:N\to M_\alpha$. Since $f_\alpha\overline{t}_\alpha h=f_\alpha m_{0\alpha}=f$, $K$ is $\lambda$-saturated.

\end{proof}

We note that coherence of $(\ck, U)$ appears to be indispensable in the ``only if'' portion of this proof.  As in the proof of the equivalence of Galois saturation and model homogeneity in \cite{Ba}, or in the more straight-forwardly category-theoretic proof of that fact in \cite{G}, coherence is the only guarantee that the newly-constructed map of underlying sets, $t_\alpha$, is a $\ck$-morphism.  This is true more broadly: when attempting to build a $\ck$-morphism element by ele\-ment, it seems 
that one must, as a rule, assume coherence to guarantee success.

This is significant, given the essential role played by the analogue of \ref{prop6.2} in AECs. The equivalence of Galois-saturated 
and model-ho\-mo\-ge\-neous models leads to uniqueness of Galois-saturated models in each car\-di\-na\-li\-ty, a result which features heavily in the existing categoricity transfer results for AECs.

\section{Stability}

We have now observed that in any accessible category $\ck$ with concrete directed colimits, it is possible to make sense of Galois types, hence we may also speak meaningfully of Galois stability, which, again, we define in the obvious way:  

\begin{defi}{\em Let $(\ck,U)$ be an accessible category with concrete directed colimits.  We say that an object $M$ in $\ck$ is \textit{$\mu$-Galois stable} if for any $\mu$-presentable object $M_0$ and morphism $M_0\to M$, there are fewer than $\mu$ types over $M_0$ realized in $M$.  We say that $\ck$ itself is \textit{$\mu$-Galois stable} if every $M$ in $\ck$ is $\mu$-Galois stable.}\end{defi}

\begin{rem}\label{re7.2}
{
\em
We can reformulate this definition by saying that for any object $M_0$ of size $\mu$ and a morphism $M_0\to M$, there are $\leq\mu$ types over $M_0$
realized in $M$. If morphisms of $\ck$ are concrete monomorphisms and $a,b\in UM_0$ are distinct elements then the types $(\id_{M_0},a)$ and $(\id_{M_0},b)$
are not equivalent. Since they are realized in $M$, there are $\mu$ types over $M_0$ realized in $M$. Thus our definition coincides with that for abstract
elementary classes. 
}
\end{rem}
 
Having generalized to a large accessible category with concrete directed colimits $(\ck,U)$, it is rather surprising that one can develop a nontrivial stability theory.  This is thanks in large part to the existence even in this context of an Ehrenfeucht-Mostowski functor, $E:\Lin\to\ck$, which is faithful, preserves directed colimits and, moreover, preserves sizes for sufficiently large $\lambda$ (see Remark \ref{re2.7}).  We denote by $\lambda_E$ the cardinal at which $E$ 
begins preserving sizes.   

Following the argument in \cite{Ba1}, which first isolated specific applications of EM-models to stability-theoretic issues in AECs, we are able to prove the following test result.  Following \cite{R}, we define:

\begin{defi}\label{def7.3}
{\em
Let $\lambda$ be an infinite cardinal. We say that an accessible category $\ck$ with directed colimits is $\lambda$-\textit{categorical} if it has, up to isomorphism, precisely one object size $\lambda$.
}
\end{defi}

\begin{theo}\label{theo7.4} Suppose there is a proper class of cardinals $\lambda$ with $\lambda^{<\lambda}=\lambda$.  Let $(\ck,U)$ be a large accessible category with concrete directed colimits such that $U$ reflects split epimorphisms. If $\ck$ is $\lambda$-categorical, then $\ck$ is $\mu$-Galois stable for all 
$\lambda_E+\lambda_U\leq\mu\leq\lambda$.\end{theo}

We proceed by way of two definitions and a handful of minor lemmas.

\begin{defi}\label{def7.5}{\em Given a composition $M_0\to\bar{M}\to M$, we say that $\bar{M}$ is \textit{$\mu$-universal over $M_0$ in $M$} if for any other factorization $M_0\to N\to M$ with $N$ $\mu$-presentable, $N$ maps in $\bar{M}$ over $M_0$, i.e. there is a morphism $N\to \bar{M}$ so that the left half of the following diagram commutes:
$$
\xymatrix@=2pc{
M_0\ar[r]\ar[dr] & \bar{M}\ar[r] & M \\
 & N\ar@{-->}[u]\ar[ur] & 
}
$$}
\end{defi}

\begin{defi}\label{def7.6}{\em We say that an object $M$ of size $\lambda$ is \textit{brimful} if for any $\mu\leq\lambda$ and $\mu$-presentable $M_0$,
any morphism $M_0\to M$ factors as $M_0\to\bar{M}\to M$ where $\bar{M}$ is $\mu$-universal over $M_0$ in $M$.}\end{defi}

We note that this is the notion from \cite{Ba}---only subtly different from the brimmed models in \cite{S1}.

\begin{propo}\label{prop7.7}If a linear order $I$ is brimful, so is $E(I)$.\end{propo}
\begin{proof} Let $M=E(I)$, with $I$ brimful.  Say $M_0$ is $\mu$-presentable, with $\lambda_E\leq\mu<\lambda$ and a morphism $a:M_0\to M$.  Since $\Lin$ is finitely accessible, it is accessible in every regular cardinal, including $\mu$.  This means that $I$ can be expressed as a $\mu$-directed colimit of $\mu$-presentable objects, $I=\colim_{\alpha\in D} I_\alpha$.  Since $\mu\geq\lambda_E$, $E$ preserves both directed colimits and presentability ranks, meaning that $M=\colim_{\alpha\in D}M_\alpha$, with $M_\alpha=E(I_\alpha)$ $\mu$-presentable for all $\alpha\in D$.

The map $M_0\to M$ factors through some $M_\alpha=E(I_\alpha)$, say as $M_0\stackrel{a_1}{\to}M_\alpha\stackrel{a_2}{\to}M$.  As $I$ is brimful, there is a $\mu$-presentable extension $\bar{I}$ of $I_\alpha$ contained in $I$ that is $\mu$-universal over $I_\alpha$ in $I$ among linear orders.  This gives an induced factorization:
$$
\xymatrix@=2pc{
M_0\ar[rrr]^{a}\ar[dr]_{a_1} & & & M \\
 & E(I_\alpha)\ar[r]_j & E(\bar{I})\ar[ur]_k &  
}
$$
Set $\bar{M}=E(\bar{I})$.  It suffices to show that $\bar{M}$ is $\mu$-universal over $M_\alpha$ in $M$.  Let $M_0\stackrel{b_1}{\to}N\stackrel{b_2}{\to}M$ be a factorization of $M_0\stackrel{a}{\to}M$ with $N$ $\mu$-presentable.  Let $\beta\in D$ be such that $\beta>\alpha$ and the map $N\stackrel{b_2}{\to}M$ factors through $M_\beta=E(I_\beta)$, say as $N\stackrel{g}{\to} M_\beta\stackrel{h}{\to}M$.  As $I$ is brimful, there is an embedding $I_\beta$ of $\bar{I}$ over $I_\alpha$, i.e. so that the following diagram commutes, as, of course, does the induced triangle in $\ck$:
$$
\xymatrix@=2pc{
I_\alpha\ar[rr]\ar[dr] & & \bar{I} & & E(I_\alpha)\ar[rr]^j\ar[dr]_f & & E(\bar{I})\\
 & I_\beta\ar[ur] & & & & E(I_\beta)\ar[ur]_i &  
}
$$
The result is the following diagram, all triangles of which commute:
$$
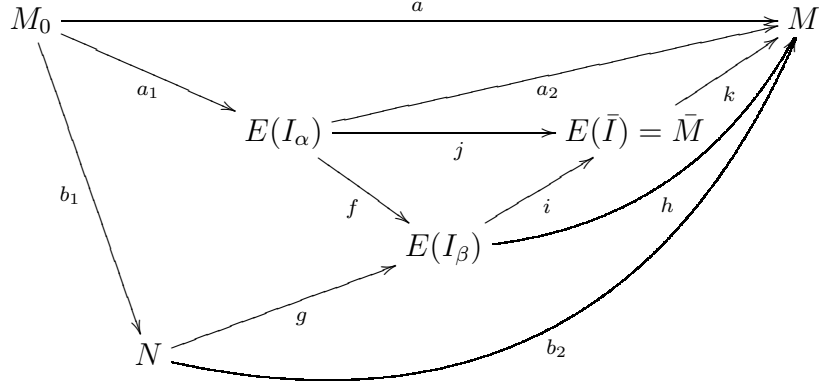
\xymatrix@=2pc{
M_0\ar[rrrrr]^a\ar[drr]_{a_1}\ar[dddr]_{b_1} & & & & & M \\
 & & E(I_\alpha)\ar[dr]_{f}\ar[rr]_{j}\ar[rrru]_{a_2} & & E(\bar{I})=\bar{M}\ar[ur]_{k} & \\
 & & & E(I_\beta)\ar[ur]_{i}\ar@/_2pc/[uurr]_{h} & & \\
 & N\ar[urr]_{g}\ar@/_5pc/[rrrruuu]_{b_2} & & & &  
}
$$
It is an exercise in diagram chasing to show that, given $b_2\circ b_1=a$, $i\circ g$ is the desired embedding of $N$ into $\bar{M}$ over $M_0$.
\end{proof}

\begin{lemma}\label{le7.8}If $\ck$ is $\lambda$-categorical, the unique object $M$ of size $\lambda$ is $\mu$-Galois stable for all $\lambda_E+\lambda_U\leq\mu<\lambda^+$.\end{lemma}
\begin{proof}Take $I=\lambda^{<\omega}$, which is brimful as a linear order by Claim 4.4 in \cite{Ba1}.  Consequently, the object $E(I)$ is brimful in $\ck$.  Given that it has size $\lambda$ (as $E$ preserves sizes for $\lambda\geq\lambda_E$), we may take $M=E(I)$.  Take $M_0$ $\mu$-presentable, and $f:M_0\to M$ a morphism.  Since $M$ is brimful, we may factor $f$ as $M_0\stackrel{f_1}{\to}\bar{M}\stackrel{f_2}{\to}M$, with $\bar{M}$ $\mu$-presentable and $\mu$-universal over $M_0$.  

Any type realized over $M_0$ in $M$ will be of the form $(f,a)$, with $a\in U(M)$.  In fact, we may factor $f$ as $M_0\stackrel{g_1}{\to}M_1\stackrel{g_2}{\to}M$ with $M_1$ $\mu$-presentable and $b\in U(M_1)$ such that $U(g_2)(b)=a$.  Hence there is a morphism $h:M_1\to\bar{M}$ such that $hg_1=f_1$.  This morphism, and the pair $(f_1,Uh(b))$, witness that our type is realized in $\bar{M}$.  Since $\mu>\lambda_U$,  $U(\bar{M})$ contains fewer than $\mu$ elements.  So, in fact, there are fewer than $\mu$ types realized in $\bar{M}$ over the subobject $M_0\stackrel{f_1}{\to}\bar{M}$.  We have shown that every type over $M_0\stackrel{f}{\to} M$ in $M$ is equivalent to one of this form: the result follows.\end{proof}

We are now in a position to prove Theorem \ref{theo7.4}:
\begin{proof} Suppose that $\ck$ is not $\mu$-stable for some $\lambda_E+\lambda_U\leq\mu<\lambda^+$.  Then there is a $\mu$-presentable $M_0$ and $f:M_0\to N$ so that there are at least $\mu$ types realized in $N$ via $f$, say $\{(f,a_\alpha)\,|\,\alpha<\mu\}$.  Since $\lambda_\ck\leq\lambda_U\leq\mu$, $\ck$ is 
$\mu^+$-accessible. Consequently, $f$ factors through a $\mu^+$-presentable object $M_1$ with the additional property that $a_\alpha\in U(M_1)$ for all $\alpha<\mu$.  
In this way, we have produced a $\mu^+$-presentable model $M_1$ that is not $\mu$-stable. By the joint embedding property, there is an object $K$ in $\ck$ 
and morphisms $u:M_1\to K$ and $v:M\to K$. Since $\ck$ is $\mu^+$-accessible, $K$ is $\mu^+$-directed colimit of $\mu^+$-presentable objects, i.e. objects of size $\mu$.  All 
these objects are isomorphic to $M$ and, because $M_1$ is $\mu^+$-presentable, $u$ factors through one of those objects. Thus $M$ is not $\mu$-stable, which
contradicts Lemma \ref{le7.8}.
\end{proof}

\begin{coro}\label{cor7.9} 
Suppose there is a proper class of cardinals $\lambda$ with $\lambda^{<\lambda}=\lambda$.  Let $\ck$ be a coherent large accessible category with concrete directed colimits whose morphisms are concrete monomorphisms. Let $\ck$ be $\lambda$-categorical for a regular cardinal $\lambda\geq\lambda_U+\lambda_E$.  Then 
the unique $M$ of size $\lambda$ is saturated.\end{coro}
\begin{proof}  Fixing $M_0$ $\mu$-presentable with $\mu<\lambda^+$ and a morphism of $M_0$ into the monster object $L$, we build a smooth chain 
of length $\lambda$ of models $M_\alpha$, where each models realizes all Galois types over its predecessor. The colimit $M_\lambda$ of this chain 
is $\lambda^+$-presentable and is $\lambda$-Galois saturated. Following \ref{prop6.2}, $M_\lambda$ is $\lambda$-saturated and, following \ref{re3.3},
$M_\lambda$ has size $\lambda$. Thus $M\cong\ M_\lambda$ is saturated.
\end{proof}

In fact, this argument gives saturated models in all regular cardinals $\mu<\lambda^+$ with $\mu\geq\lambda_U+\lambda_E$. It also yields
the following.

\begin{rem}\label{re7.9}
{
\em
Suppose there is a proper class of cardinals $\lambda$ with $\lambda^{<\lambda}=\lambda$.  Let $\ck$ be a coherent large accessible category with concrete directed colimits, 
whose morphisms are concrete monomorphisms. 

(1) Let $\ck$ be $\lambda$-Galois stable for a regular cardinal $\lambda\geq\lambda_U+\lambda_E$. Then $\ck$ has a saturated object of size $\lambda$.

(2) Let $\lambda\geq\lambda_U+\lambda_E$ be a regular cardinal. Then $\ck$ is $\lambda$-categorical if and only if every object of $\ck$ of size $\geq\lambda$
is $\lambda$-saturated. From the proof of theorem 7.3 we know that any object of size $\geq\lambda$ is a $\lambda^+$-directed colimit of objects of size $\lambda$.
Since a $\lambda^+$-directed colimit of $\lambda$-saturated objects is $\lambda$-saturated, Corollary 7.8 implies that all objects of size $\geq\lambda$ 
are $\lambda$-saturated provided that $\ck$ is $\lambda$-categorical. Conversely, if all objects of size $\geq\lambda$ are $\lambda$-saturated, all objects
of size $\lambda$ are saturated and thus isomorphic. Since $\ck$ has an object of size $\lambda$, it is $\lambda$-categorical.
}
\end{rem}


\begin{thebibliography}{EAPT}
\itemsep=2pt
 
\bibitem{AHT} J. Ad\' amek, H. Hu and W. Tholen, {\em On pure morphisms in accessible categories}, J. Pure Appl. Algebra
107 (1996), 1-8.
 
\bibitem{AR} J. Ad\'{a}mek and J. Rosick\'{y}, {\em Locally Presentable and Accessible Categories}, Cambridge University
Press 1994.
 
\bibitem{Ba} J. Baldwin, {\em Categoricity}, AMS 2009.

\bibitem{Ba1} J. Baldwin, {\em Ehrenfeucht-Mostowski models in abstract elementary classes}.  Logic and Its Applications (Y. Zhang and A. Blass, eds.), vol. 380, Contemporary Mathematics.

\bibitem{BaS} J. Baldwin and S. Shelah, {\em Examples of non-locality}, J. Symbolic Logic 73 (2008), 765-782.

\bibitem{BS1} J. Baldwin and S. Shelah, {\em A Hanf number for saturation and omission}, Fund. Math. 213 (2011), 255-270.

\bibitem{BR} T. Beke and J. Rosick\'y, {\em Abstract elementary classes and accessible categories}, Annals Pure Appl. Logic 163 (2012), 2008-2017.
 
\bibitem{B} W. Boney, Tameness from large cardinal axioms, arXiv:1303.0550 
 
\bibitem{G} R. Grossberg, {Classification theory for abstract elementary classes}, Logic
and Algebra (Yi Zhang, ed.), vol. 302, AMS 2002.
 
\bibitem{GV} R. Grossberg and M. VanDieren, {\em Categoricity from one successor cardinal in tame abstract elementary classes}, Journal of Math. Logic, vol. 6, no. 2 (2006) 181--201

\bibitem{HK} T. Hyttinen and M Kes\"{a}l\"{a}, {\em Independence in finitary abstract elementary classes}, Annals Pure Appl. Logic 143 (2006), 103-138.

\bibitem{J} T. Jech {\em Set Theory}, Academic Press 1978.

\bibitem{L} M. Lieberman, {\em Category theoretic aspects of abstract elementary classes}, Annals Pure Appl. Logic 162 (2011), 903-915.
 
 
\bibitem{K} D. W. Kueker, {\em Abstract elementary classes and infinitary logic}, Annals Pure Appl.Logic 156 (2008), 274-286.
 
\bibitem{MP} M. Makkai and R. Par\' e, {\em Accessible Categories: The Foundations of Categorical Model Theory}, AMS 1989.

\bibitem{MS} M. Makkai and S. Shelah, {\em Categoricity of theories in $L_{\kappa,\omega}$ with $\kappa$ a compact cardinal},
Annals Pure Appl. Logic 47 (1990), 41-97.

\bibitem{R0} J. Rosick\'{y}, {\em Concrete categories and infinitary languages}, J. Pure Appl. Alg. 22 (1981), 309-339.

\bibitem{R} J. Rosick\'y, {\em Accessible categories, saturation and categoricity}, J. Symb. Logic 62 (1997), 891-901. 
  
\bibitem{R1} J. Rosick\'y, {\em On combinatorial model categories}, Appl. Cat. Struct. 17 (2009), 303-316.          

\bibitem{S1} S. Shelah, {\em Categoricity of abstract elementary classes: going up inductive step}, preprint 600.

\bibitem{S} S. Shelah, {\em Classification of nonelementary classes II, abstract elementary classes}, Lecture Notes in Math. 1292 (1987),
419-497.

\end{thebibliography}
\end{document}